\documentclass[11pt]{amsart}
\usepackage{fullpage}

\newtheorem{theorem}{Theorem}[section]

\newtheorem{proposition}{Proposition}[section]
\newtheorem{lemma}{Lemma}[section]
\newtheorem{definition}{Definition}[section]

\usepackage{mathrsfs}

\def \bC {\mathbb C}

\def \bH {\mathbb H}

\def \bN {\mathbb N}

\def \bR {\mathbb R}
\def \bS {\mathbb S}

\def \cA {\mathcal A}

\def \cC {\cC}

\def \cF {\mathcal F}

\def \cH {\mathcal H}

\def \cK {\mathcal K}
\def \cL {\mathcal L}
\def \cM {\mathcal M}

\def \cS {\mathcal S}

\def \fg {\mathfrak g}

\def \fU {\mathfrak U}

\def \sL{\mathscr L}

\def \tr {{\rm Tr}}

\def \id {\text{\rm I}}
\def \Op {{\rm Op}}

\def \Gh {{\widehat G}}
\def \eps {\varepsilon}

\title{Towards semi-classical analysis for sub-elliptic operators
}

\author[V. Fischer]{V\'eronique Fischer}
\address[V. Fischer]%
{University of Bath, Department of Mathematical Sciences, Bath, BA2 7AY, UK} 
\email{v.c.m.fischer@bath.ac.uk}

\begin{document}

\begin{abstract}
We discuss the recent developments of semi-classical and micro-local analysis 
in the context of nilpotent Lie groups and for sub-elliptic operators. 
In particular, we give an overview  of pseudo-differential calculi recently defined on nilpotent Lie groups as well as of the notion of quantum limits in the Euclidean and nilpotent cases. 

\medskip \noindent {\sc{2010 MSC.}}
43A80; 58J45,  35Q40.

 \noindent {\sc{Keywords.}} 
Analysis on nilpotent Lie groups,
	evolution of solutions to the Schr\"odinger equation, 
	micro-local and semi-classical analysis for sub-elliptic operators, abstract harmonic analysis, $C^*$-algebra theory.
\end{abstract}
\maketitle

\makeatletter
\renewcommand\l@subsection{\@tocline{2}{0pt}{3pc}{5pc}{}}
\makeatother

\tableofcontents

\section{
Introduction}

Since the 1960's, the analysis of elliptic operators has made fundamental progress with the emergence of pseudo-differential theory and the subsequent developments of micro-local and semi-classical analysis. 
In this paper, we consider some questions that are well understood for elliptic operators
and we discuss analogues in the setting  of sub-elliptic operators.

\subsection{The questions in the elliptic framework.}
The questions we are interested in concern   the tools that have been developed in the elliptic framework to describe and  understand   the limits in space or in phase-space of families of functions.
They are of two natures: micro-local and semi-classical. 
Micro-local analysis aims at understanding elliptic operators in high frequency, 
while semi-classical analysis  investigate the mathematical evolution of functions and operators depending on a small parameter $\eps$ (akin to the Planck constant in quantum mechanics) that goes to zero. 

A typical micro-local question is, for instance,  to `understand the convergence' as $j\to \infty$ of an orthonormal basis of  eigenfunctions $\psi_j$, $j=0,1,2,\ldots$
$$
\Delta \psi_j = \lambda_j \psi_j, \qquad
\qquad\mbox{with}\quad 0=\lambda_0 <\lambda_1\leq \lambda_2\leq \ldots
$$
 of the  Laplace operator $\Delta$ on a compact Riemannian manifold $M$. 
One way to answer this question is to describe the accumulation points of the sequence of probability measures $|\psi_j(x)|^2 dx$, $j=0,1,2,\ldots$ 
If $M$ is the $n$-dimensional torus or if the geodesic flow of $M$ is ergodic, then 
 the volume element $dx$ is an accumulation point of  $|\psi_j(x)|^2 dx$, $j=0,1,2,\ldots$  and   one can extract a subsequence of density one $(j_k)_{k\in \bN}$, 
$$
\mbox{i.e.}\quad	\lim_{\Lambda\to \infty} \frac{|\{j_k : \lambda_{j_k} \leq \Lambda\}|}{|\{j : \lambda_{j} \leq \Lambda\}|} =1,
$$
 for which the convergence holds, that is,  for any continuous function $ a:M\to \bC$,
\begin{equation}
\label{eq_QE_intro}
\lim_{k\to +\infty}
	\int_M a( x) \ | \varphi_{j_k}(x)|^2 d x = \int_M a ( x)\, d x.
\end{equation}
Under the ergodic hypothesis,
this is a famous result due to 
Shnirelman \cite{Shnirelman}, Colin de Verdi\`ere \cite{Colin85}, and Zelditch \cite{zelditch} in 1970's and 80's and sometimes called the Quantum Ergodicity Theorem - see also the semi-classical analogue in
\cite{helffer+martinez+robert}.

A typical semi-classical problem is to understand the quantum evolution of the Schr\"odinger equation
$$
i\eps \partial_t \psi^\eps = -\frac {\eps^2}2\Delta \psi^\eps, 
\qquad \mbox{given an} \ L^2\mbox{-bounded family of initial datum}\ \psi^\eps|_{t=0} = \psi^\eps_0;
$$ 
in this introduction, let us consider the setting of  $\bR^n$ to fix  ideas.
Again, a mathematical formulation consists in describing the accumulation points of the sequence of  measures $|\psi^\eps(t,x) |^2 dx dt$ as $\eps\to 0$. 

\subsection{Sub-elliptic operators.}
In this paper, we discuss the extent to which these types of questions have been addressed for sub-elliptic operators. 
The main examples of sub-elliptic operators are sub-Laplacians $\cL$ generalising the Laplace operator. 
Concrete examples of sub-elliptic and non-elliptic operators include 
$$
\cL_{G} = -\partial_u^2 - (u\partial_v)^2 \quad\mbox{on}\ \bR_u\times \bR_v =\bR^2 ,
$$
often called the Grushin operator (the subscript $G$ stands for Grushin). 
More generally, H\"ormander sums of squares  are sub-elliptic operators; they are operators $\cL= -X_1^2-\ldots -X_{n_1}^2-X_0$ on a manifold $M^n$ where the vector fields $X_j$'s together with their iterated brackets generate  the tangent space $TM$ at every point \cite{hormander67}. 
A more geometric source of sub-Laplacians is the analysis on sub-Riemannian manifolds, starting with CR manifolds such as the unit sphere of the complex plane $\bC^2$ or even of $\bC^n$ for any $n\geq 2$, and more generally contact manifolds. 
Well-known contact manifolds of  dimension three  include  the Lie group $SO(3)$ with two of its three canonical vector fields, as well as the motion group $\bR^2_{x,y} \times \bS^1_\theta$ with the vector fields $X_1=\cos \theta \partial_x +\sin \theta \partial_y$, and $X_2=\partial_\theta$. 
Sub-Laplacians appear in many parts of sciences, in physics, biology, finance, etc., see \cite{bramanti}.

A particular framework of sub-Riemannian and sub-elliptic settings is given by Carnot groups; the latter are stratified nilpotent Lie groups $G$ equipped with a basis $X_1,\ldots, X_{n_1}$  for the first stratum $\fg$ of the Lie algebra of $G$. 
Using the natural identification of $\fg$ with the space left-invariant vector fields,
the canonical sub-Laplacian is then $\cL= -X_1^2-\ldots -X_{n_1}^2$.
This is an important class of examples not only because this provides a wealth of models and settings on which to test conjectures, but also more fundamentally, as any H\"ormander sum of squares can be lifted -  at least theoretically 
\cite{FollandStein74,Rothschild+Stein,Nagel+Stein+Wainger}
 -  to a Carnot group. 
For instance, the Grushin operator $\sL_G$ on $\bR^2$ described above can be lifted to 
the sub-Laplacian $\sL_{\bH_1} = -X_1^2 -X_2^2$ on 
the Heisenberg group $\bH_1$;
here the product on $\bH_1\sim \bR^3_{x,y,t}$ is given by 
$$
(x,y,t) (x',y',t') = \left(x+x',y+y',t+t'+\frac 12 (xy'-x'y)\right), 
$$
and $X_1= \partial_x -\frac y2 \partial_t$ and $X_2= \partial_y +\frac x2 \partial_t$.

The examples given above infer that the analysis of sub-elliptic operators such as H\"ormander sums of squares is more non-commutative than in the elliptic case. Indeed, the commutator of the vector fields $X_j$'s  in our examples above  usually produces terms that cannot be neglected in any meaningful elliptic analysis, whereas in the elliptic case the $X_j$'s can be chosen to yield local coordinates and therefore commute up to lower order terms.  
This led to a difficult non-commutative analysis in the late 70's and 80's around the ideas of lifting the nilpotent Lie group setting \cite{FollandStein74,Rothschild+Stein,Nagel+Stein+Wainger}, 
and subsequently in 80's and 90's using  Euclidean micro-local tools as well  \cite{Fefferman+Phong,Sanchez,Parmeggiani}. 
At the same time, sub-Riemannian geometry was emerging. Although many functional features are almost identical to the 
Riemannian case \cite{Strichartz86}, there are fundamental differences regarding e.g. geodesics, charts or local coordinates, tangent spaces etc.
see e.g.
\cite{Bellaiche,Gromov,Montgomery,AgrachevBB}.  

The analysis of operators on classes of sub-Riemannian manifolds
started with CR and contact manifolds \cite{FollandStein74}, 
followed by a calculus on Heisenberg manifolds \cite{Beals+Greiner,PongeAMS2008}.
In 2010 \cite{vanErp}, an index theorem was proved for sub-elliptic operators on contact manifolds.
The key idea was to adapt Connes' tangent groupoid \cite{Connes} from the Riemannian setting to the sub-Riemannian's. 
 For contact manifolds, the Euclidean tangent space is then replaced with the Heisenberg group. Since then, 
considerable progress has been achieved in the study
of spectral properties of sub-elliptic operators in these contexts (see e.g. \cite{Dave+Haller}) with the development of these groupoid techniques on filtered manifolds \cite{vanErp,choi+ponge,vanErp+Y}. 

Few works on sub-elliptic operators followed the path opened by M. Taylor \cite{TaylorAMS} at the beginning of the 80's, 
that is,  to use the representation theory of the underlying groups to tackle the non-commutativity. 
To the author's knowledge, in the nilpotent case, they are essentially \cite{Bahouri+Fermanian+Gallagher,R+F_monograph,Bahouri+Chemin+Danchin}  
 and, surprisingly, have appeared only in the past decade. 

\subsection{Aim and organisation of the paper}
This paper describes the scientific journey of the author and of her collaborator Clotilde Fermanian-Kammerer  towards
micro-local and semi-classical analysis for sub-elliptic operators, especially  on nilpotent Lie groups. 
The starting point of the investigations was to define and study the analogues of micro-local defect measures. As explained in Section \ref{sec_QL}, this has led  to adopt the more general view point and the vocabulary from $C^*$-algebras regarding states even in the Euclidean or elliptic case. 
The first results regarding micro-local defect measure and semi-classical measures on nilpotent Lie groups are presented 
in Section \ref{sec_PDOQLG}, including 
applications and future works.

\subsection{Acknowledgement}
The author is grateful to the Leverhulme Trust for their support via Research Project Grant 2020-037. 

This paper summarises  the main ideas  discussed by the author to the Bruno Pini Mathematical Analysis Seminar of the University of Bologna in May 2021. 
The author would like to thank the organisers for giving her the opportunity  to present the project and for their warm welcome - even in zoom form.

\section{Quantum limits in Euclidean or elliptic settings}
\label{sec_QL}

In this section, we discuss how micro-local defect measures and semi-classical measures can be seen as quantum limits, that is, as states of $C^*$-algebras. 

\subsection{Micro-local defect measures}

The notion of micro-local defect measure, also called H-measure, emerged around 1990 independently in the works of P. G\'erard \cite{gerard_91} and L. Tartar \cite{tartar}. 
Their motivations came from PDEs, in relation to the div-curl lemma and more generally to describe phenomena of compensated compactness. The following result gives the existence of micro-local defect measures:

\begin{theorem}[\cite{gerard_91}]
\label{thm_MDM}
Let $\Omega$ be an open subset of $\bR^n$. 
Let $(f_j)_{j\in \bN}$ be a bounded sequence in $L^2(\Omega,loc)$ converging weakly to 0.  
Then there exist a subsequence $(j_k)_{k\in \bN}$ and a positive Radon measure $\gamma$ on $\Omega\times \bS^{n-1}$ such that the convergence
$$
(Af_j,f_j)_{L^2} \longrightarrow_{j=j_k, k\to \infty} \int_{\Omega\times \bS^{n-1}} a_0(x,\xi) d\gamma(x,\xi)
$$
holds for all classical pseudo-differential operator $A$, $a_0$ denoting its principal symbol. 
\end{theorem}

Here, $\bS^{n-1}$ denotes the unit sphere in $\bR^n$.
The classical pseudo-differential calculus refers to all the H\"ormander pseudo-differential operators of non-positive order, with symbols admitting a homogeneous expansion and with integral kernel compactly supported in $\Omega\times \Omega$. 

The measure $\gamma$ in Theorem \ref{thm_MDM}
is  called a micro-local defect measure for $(f_j)$, 
or the (pure) micro-local defect measure for $(f_{j_k})$. 
Examples of micro-local defect measures include 
\begin{itemize}
\item an  $L^2$-concentration in space $f_j(x) = j^{n/ 2}\chi(j(x-x_0) ) $ about a point $x_0$ (here, $\chi\in C_c^\infty(\bR^n)$ is some given  function), whose micro-local defect measure is $\gamma(x,\xi)=\delta_{x_0} (x)\otimes d_\chi(\xi) d\sigma(\xi)$, where $\sigma$ is the uniform measure on $\bS^{n-1}$ (i.e. the standard surface measure on the unit sphere) and $d_\chi(\xi):=  \int_{r=0}^\infty |\widehat \chi(r\xi)|^2 r^{n-1} dr$ , 
\item an  $L^2$-concentration in oscillations $f_j(x) = \psi(x) e^{2i\pi j \xi_0 \cdot x}  $ about a frequency  $\xi_0\in \bS^{n-1}$ (here, $\psi$ is some given smooth function with compact support in $\bR^n$), whose micro-local defect measure is $\gamma(x,\xi)=|\psi(x)|^2dx  \otimes \delta_{\xi_0}(\xi) $.  
\end{itemize} 

Theorem \ref{thm_MDM} extends readily to manifolds by replacing $\Omega\times \bS^{n-1}$ with the spherical co-tangent bundle. 

The introduction of this paper mentions the Quantum Ergodicity Theorem, see \eqref{eq_QE_intro}. This is in fact the reduced version `in position'. A modern presentation of the full Quantum Ergodicity Theorem can be expressed as saying that the Liouville measure $dx\otimes d\sigma(\xi)$ is a micro-local defect measure of the sequence $(\psi_j)_{j\in \bN_0}$, for which the subsequence $(j_k)$ is of density one. 

\subsection{The viewpoint of quantum limits}
\label{subsec_viewptQL}

The author's definition of quantum limits is a notion along the line of the following: 
\begin{definition}
The \emph{quantum limit} of a sequence $(f_j)$ of unit vectors in a Hilbert space $\cH$ is any accumulation point 
of the functional $A\mapsto (Af_j,f_j)_\cH$ on a sub-$C^*$-algebra of $\sL(\cH)$.
\end{definition}

One may still keep the vocabulary  `quantum limits' in slightly more general contexts. 
For instance, one often encounters a subalgebra of $\sL(\cH)$ that may need to be completed into a $C^*$-algebra, 
possibly after quotienting by (a  subspace of) the kernel of the mapping   $A\mapsto \limsup_{j\to \infty} |(Af_j,f_j)_\cH|$. 
We may also consider a bounded family $(f_j)$ in $\cH$ rather than unit vectors, leaving the normalisation for the proofs of further properties.  

The applications we have in mind involve pseudo-differential calculi as subalgebras of $\sL(\cH)$ where the Hilbert space $\cH$ is some $L^2$-space. 
A quantum limit in this context will often turn out to be a state 
(or a positive functional if the $\|f_j\|_{\cH}$'s are only bounded)
 on a space of symbols, hence a positive Radon measure in the commutative case. Indeed, from functional analysis, we know that a bounded linear functional on the space of continuous functions on a (say) compact space is given by a Radon measure, and if the functional is also positive, the measure will be positive as well.

\medskip

Let us now explain how the viewpoint of quantum limits and states gives another proof of Theorem \ref{thm_MDM} by first obtaining  the following result:

\begin{lemma}
\label{lem_thm_MDM}
Let $\Omega$ be an open bounded subset of $\bR^n$. 
Let $(f_j)_{j\in \bN}$ be a bounded sequence in $L^2(\bar \Omega)$ converging weakly to 0 as $j\to \infty$.
Then there exists a subsequence $(j_k)_{k\in \bN}$ and a positive Radon measure on $\bar \Omega\times \bS^{n-1}$ such that 
$$
(Af_j,f_j)_{L^2(\bar\Omega)} \longrightarrow_{j=j_k, k\to \infty} \int_{\bar \Omega\times \bS^{n-1}} a_0(x,\xi) d\gamma(x,\xi)
$$
holds for all classical pseudo-differential operator $A$
whose principal symbol $a_0$   is $x$-supported in $\bar \Omega$. 
\end{lemma}
\begin{proof}[Sketch of the proof of Lemma \ref{lem_thm_MDM}]
If $\limsup_{j\to \infty} \|f_j\|_{L^2(\bar \Omega)}=0$, then $\gamma=0$.
Hence, we may assume that
 $\limsup_{j\to \infty} \|f_j\|_{L^2(\bar \Omega)}=1$.
We consider the  sequence of functionals $\ell_j :A \mapsto (Af_j,f_j)_{L^2}$  
on the algebra $\cA_0$ of  classical pseudo-differential operators $A$ whose symbols are $x$-supported in $\bar \Omega$. 

The weak convergence of $(f_j)$  to zero  means that $\lim_{j\to \infty} \ell_j(A)=0$ for every operator $A$ in 
$$
\cK = \{\mbox{compact operator in}  \ \cA_0 \} \sim \{\mbox{operators in} \ \cA_0 \ \mbox{of order }<0\},
$$
 by Rellich's theorem. 

The properties of the pseudo-differential calculus imply that a limit of $(\ell_j)_{j\in \bN}$  is a state on the closure of 
the quotient $\overline{\cA_0 /\cK}$; we recognise the abelian $C^*$-algebra generated by the principal symbols $x$-supported in $\bar \Omega$, that is, the space of continuous functions on the compact space $\bar\Omega\times \bS^{n-1}$. Such a state is given by a positive Radon measure on $\bar\Omega\times \bS^{n-1}$. 
\end{proof}

Let us now  give the  new proof of Theorem \ref{thm_MDM} announced above. 
Adopting the setting of the statement, we find a sequence of open sets $\Omega_k$, $k=1,2,\ldots$ such that $\bar \Omega_k$ is a compact subset of $\Omega_{k+1}$ and $\cup_{k\in \bN} \Omega_k =\Omega$.
Applying Lemma \ref{lem_thm_MDM} to each $\Omega_k$ together with a diagonal extraction yield Theorem \ref{thm_MDM}.

\medskip

The author and her collaborator Clotilde Fermanian Kammerer are forever indebted to Professor Vladimir Georgescu for his enlightening explanations on the proof of the existence of micro-local defect measures given above. 
Vladimir Georgescu's comments  describe also the states of other  $C^*$-algebras of operators bounded on  $L^2$ from profound works by O. Cordes and his collaborators  on Gelfand theory for pseudo-differential calculi
\cite{Cordes+H,Cordes79,Cordes87,Cordes95,Taylor71}. They also provide  a framework which generalises the two original  proofs of the existence of H-measure / micro-local defect measures: 
\begin{itemize}
\item 
the one by L. Tartar \cite{tartar} 
which  uses operators of multiplication  in position and Fourier multipliers in frequencies, and 
\item  
the one by P. G\'erard's \cite{gerard_91} relying on properties of the classical pseudo-differential calculus, especially the G\r arding inequality.
\end{itemize}

\subsection{Semi-classical measures as quantum limits}
\label{subsec_scmql}
 
 The semi-classical calculus used here is `basic' in the sense that it is restricted to the setting of $\bR^n$ and to operators $\Op_\eps(a)$ with $a\in C_c^\infty(\bR^n\times\bR^n)$ for instance. 
 Here $\Op_\eps(a)=\Op(a_\eps)$ is `the' pseudo-differential operator with symbol $a_\eps(x,\xi)=a(x,\eps \xi)$ via a chosen $t$-quantisation  on $\bR^n$ - for instance  the Weyl quantisation ($t=1/2$) or the Kohn-Nirenberg quantisation ($t=0$, also known as PDE quantisation and often written as $\Op(a) = a(x,D)$).
 More sophisticated semi-classical calculi can be defined, for instance allowing the symbols $a$ to depend on $\eps$ and in the context of manifolds, see e.g. \cite{zworski}.    
 
 Semi-classical measures were introduced in the 90's in works such as \cite{gerard_X,gerardleichtnam,GMMP,LionsPaul}. 
In this section, we show how the viewpoint of quantum limits gives a simple proof of their existence as in the case of micro-local defect measures (see Section \ref{subsec_viewptQL}). 
With the Weyl quantisation, the existence of semi-classical measures can be proved using graduate-level functional analysis and the resulting measures are called Wigner measures. But our proof below is independent of the chosen quantisation.

 \begin{theorem}
 \label{thm_scm}
 Let $(f_\eps)_{\eps>0}$ be a bounded family in $L^2(\bR^n)$. 
 Then there exists a sequence $\eps_k$, $k\in \bN$ with $\eps_k\to 0$ as $k\to \infty$, and a positive Radon measure  $\gamma$ on $\bR^n\times \bR^n$ such that 
 $$
 \forall a\in C_c^\infty (\bR^n\times\bR^n)
 \qquad
(\Op_\eps (a)f_\eps,f_\eps)_{L^2} \longrightarrow_{\eps=\eps_k, k\to \infty} \int_{ \bR^n\times \bR^n} a(x,\xi) d\gamma(x,\xi). 
$$
 \end{theorem}

\begin{proof}[Sketch of the proof of Theorem \ref{thm_scm}]
We may assume $\limsup_{\eps\to 0} \|f_\eps\|_{L^2}=1$.
We set $\ell_\eps(a) :=(\Op_\eps (a)f_\eps,f_\eps)_{L^2}$.  
 For each $a\in C_c^\infty (\bR^n\times\bR^n)$, $\ell_\eps(a)$ is bounded so its limits exist as $\eps\to 0$. 
 A diagonal extraction and the separability of $C_c^\infty (\bR^n\times\bR^n)$ yield the existence of $\ell = \lim_{k\to \infty} \ell_{\eps_k}$ on $C_c^\infty(\bR^n\times \bR^n)$. 
 From the properties of the  semi-classical calculus, one checks that $\ell$ extends to a state of the commutative $C^*$-algebra 
$\overline{C_c^\infty(\bR^n\times \bR^n)}$, hence a positive Radon measure on $\bR^n\times\bR^n$. 
\end{proof}

The semi-classical analogues of the examples of micro-local defect measures are: 
\begin{itemize}
\item an  $L^2$-concentration in space $f_\eps(x) = \eps^{- n/2}\chi(\frac{x-x_0}\eps ) $ about a point $x_0$ (again, $\chi\in C_c^\infty(\bR^n)$ is some given  function), whose semi-classical measure is $\gamma(x,\xi)=\delta_{x_0}(x) \otimes |\widehat \chi(\xi)|^2 d\xi$, 
\item an  $L^2$-concentration in oscillations $f_\eps(x) = \psi(x) e^{2i\pi  \xi_0 \cdot x / \eps}  $ about a frequency  $\xi_0\in \bR^{n}$ (again, $\psi\in C_c^\infty(\bR^n)$ is some given function), whose semi-classical measure is $\gamma=|\psi(x)|^2dx  \otimes \delta_{\xi_0}(\xi) $.  
\end{itemize}

\subsection{Applications}
\label{subsec_app}

Let us give an application of quantum limits to semi-classical analysis already mentioned in the introduction in the form of the following result taken from   \cite[Appendix A]{FFJST}. 
This is  an elementary version of properties   that hold in more general settings and for more general Hamiltonians, including integrable systems (see~\cite{AFM,CFM}). 

\begin{proposition}
Let $(\psi^\eps_0)_{\eps>0}$ be a bounded family in $L^2(\bR^n)$
and the associated solutions to the Schr\"odinger equation,
$$
i\eps^\tau \partial_t \psi^\eps = -\frac {\eps^2}2\Delta \psi^\eps, 
\qquad \ \psi^\eps|_{t=0} = \psi^\eps_0.
$$ 
where  $\Delta=- \sum_{1\leq j\leq n} \partial_{x_j}^2$ is the standard Laplace operator on $\bR^n$.
We assume that the oscillations of the initial data are exactly of size $1/\eps$ in the sense that we have:
$$
\exists s,C_s>0,\qquad \forall \eps>0\qquad   \eps^{s} \| \Delta^{s / 2} \psi^\eps_0\|_{L^2(\bR^n)}+  \eps^{-s} \| \Delta^{-{s / 2} }\psi^\eps_0\|_{L^2(\bR^n)}\leq C_s. 
$$
Any limit point of the measures $\left|\psi^\eps(t,x)\right| ^2 dxdt$ as $\eps\to 0$  is of the form $\varrho_t(x) dt$ where $\varrho_t$ is a measure on $\bR^n$ satisfying:
\begin{enumerate}
\item $\partial_t \varrho_t =0$ for $\tau\in(0,1)$,  
\item $\varrho_t(x)=\int_{\bR^n} \gamma_0(x-t\xi,d\xi)$ for $\tau=1$,
\item  $\varrho_t=0$ for $\tau >1$. 
\end{enumerate} 
\end{proposition}

\begin{proof}
Using for instance the notion of quantum limits, we obtain  time-dependent semi-classical measures in the sense of the existence of a subsequence $(\eps_k)$ and of a continuous map $t\mapsto \gamma_t$ from $\bR$ to the space of positive Radon measures such that 
$$
\int_\bR \theta(t) \,
(\Op_\eps (a)f_\eps,f_\eps)_{L^2} \longrightarrow_{\eps=\eps_k, k\to \infty} \iint_{\bR \times  \bR^{2n}} \theta(t)\, a(x,\xi)\, d\gamma_t(x,\xi) dt,
$$
for any $\theta\in C_c^\infty(\bR)$ and $a\in C_c^\infty(\bR^{2n})$. 
Now, 
up to a further extraction of a subsequence, we obtain using  the Schr\"odinger equation:
\begin{enumerate} 
\item for $\tau\in(0,1)$,  $\gamma_t(x,\xi)=\gamma_0(x,\xi)$ for all times $t\in\bR$,
\item for $\tau=1$,  
$
\partial_t \gamma_t(x,\xi) = \xi\cdot \nabla_x \gamma_t(x,\xi)$
in the sense of distributions,
\item for $\tau>1$, 
$\gamma_t=0$ for all times $t\in\bR$.
\end{enumerate}
Taking the $x$-marginals of the measures $\gamma_t$ gives the measures  described in the statement.  
\end{proof}

The usual Schr\"odinger equation corresponds to $\tau=1$, as  in the introduction of this paper. 
In this case, the description of the semi-classical measure above provides the link between the quantum world and the classical one: $\gamma_t$ is the composition of $\gamma_0$ with the Hamiltonian flow from classical mechanics. 

\section{Pseudo-differential theory and quantum limits on nilpotent Lie groups}
\label{sec_PDOQLG}

In this section, we will present the works 
\cite{FFchina,FFPisa,FFJST} of Clotilde Fermanian-Kammerer and the author about quantum limits on nilpotent Lie groups. We will only describe briefly the setting and the notation, referring the interested reader to the literature for all the technical details. 
We will end with a word on future  developments. 

\subsection{Preliminaries on nilpotent Lie groups}

Let us consider a nilpotent Lie group $G$; we will always assume that nilpotent Lie groups are connected and simply connected. 
If we fix a basis $X_1,\ldots, X_n$ of its Lie algebra $\fg$, 
via the exponential mapping $\exp_G : \fg \to G$, we   identify 
the points $(x_{1},\ldots,x_n)\in \bR^n$ 
 with the points  $x=\exp_G(x_{1}X_1+\cdots+x_n X_n)$ in~ $G$.
This also leads to a corresponding Lebesgue measure on $\fg$ and the Haar measure $dx$ on the group $G$,
hence $L^p(G)\cong L^p(\bR^n)$ 
and  we allow ourselves to denote by $C_c^\infty(G), \, \cS(G)$ etc,
the spaces of continuous functions, of smooth and compactly supported functions or 
of Schwartz functions on $G$ identified with $\bR^n$,
and similarly for distributions.

The group convolution of two functions $f_1$ and $f_2$, 
for instance square integrable, 
is defined via 
$$
 (f_1*f_2)(x):=\int_G f_1(y) f_2(y^{-1}x) dy.
$$
The convolution is not commutative: in general, $f_1*f_2\not=f_2*f_1$.

A vector of $\fg$ defines a left-invariant vector field on $G$ 
and, more generally, 
 the universal enveloping Lie algebra $\fU(\fg)$ of $\fg$ 
is isomorphic to the space of the left-invariant differential operators; 
we keep the same notation for the vectors and the corresponding operators. 

Let $\pi$ be a representation of $G$. 
Unless otherwise stated, we always assume that such a representation $\pi$ 
is strongly continuous and unitary, and acts on a separable Hilbert space denoted by $\cH_\pi$.
Furthermore, we keep the same notation for the corresponding infinitesimal representation
which acts on  $\fU(\fg)$ 
and on the space $\cH_\pi^\infty$ of smooth vectors. 
It is characterised by its action on $\fg$
$$
\pi(X)=\partial_{t=0}\pi(e^{tX}),
\quad  X\in \fg.
$$
We define the \emph{group Fourier transform} of a function  $f\in L^1(G)$
  at $\pi$ by
$$
 \pi(f) \equiv \widehat f(\pi) \equiv \cF_G(f)(\pi)=\int_G f(x) \pi(x)^*dx.
$$ 

We denote by $\Gh$ the unitary dual of $G$,
that is, the unitary irreducible representations of $G$ modulo equivalence and identify a unitary irreducible representation 
with its class in $\Gh$. The set $\Gh$ is naturally equipped with a structure of standard Borel space.
The Plancherel measure is the unique positive Borel measure $\mu$ 
on $\Gh$ such that 
for any $f\in C_c(G)$, we have:
\begin{equation}
\label{eq_plancherel_formula}
\int_G |f(x)|^2 dx = \int_{\Gh} \|\cF_G(f)(\pi)\|_{HS(\cH_\pi)}^2 d\mu(\pi).
\end{equation}
Here $\|\cdot\|_{HS(\cH_\pi)}$ denotes the Hilbert-Schmidt norm on $\cH_\pi$.
This implies that the group Fourier transform extends unitarily from 
$L^1(G)\cap L^2(G)$ to $L^2(G)$ onto 
$L^2(\Gh):=\int_{\Gh} \cH_\pi \otimes\cH_\pi^* d\mu(\pi)$
which we identify with the space of $\mu$-square integrable fields on $\Gh$.

A \emph{symbol} is a measurable field of operators $\sigma(x,\pi):\cH_\pi^\infty \to \cH_\pi^\infty$, parametrised by $x\in G$ and $\pi\in \Gh$.
We formally associate to $\sigma$ the operator $\Op(\sigma)$
as follows
$$
\Op(\sigma) f (x) := \int_G 
\tr \left(\pi(x) \sigma(x,\pi) \widehat f (\pi) \right)
d\mu(\pi),
$$
where $f\in \cS(G)$ and $x\in G$.
If $G$ is the abelian group $\bR^n$, this corresponds to the Kohn-Nirenberg quantisation.

Regarding symbols, when no confusion is possible,
we will allow ourselves some notational shortcuts, 
for instance writing $\sigma(x,\pi)$ 
when considering the field of operators $\{\sigma(x,\pi) :\cH_\pi^\infty \to \cH_\pi^\infty, (x,\pi)\in G\times\Gh\}$ with the usual identifications
for $\pi\in \Gh$ and $\mu$-measurability.

This quantisation has already been observed in \cite{TaylorAMS,Bahouri+Fermanian+Gallagher,R+F_monograph} for instance.
It can be viewed as an analogue of the Kohn-Nirenberg quantisation
since the inverse formula can be written as 
$$
 f (x) := \int_G 
\tr \left(\pi(x) \widehat f (\pi) \right)
d\mu(\pi),
\quad f\in \cS(G), \ x\in G.
$$
This also shows that the operator 
associated with the symbol $\id=\{\id_{\cH_\pi} , (x,\pi)\in G\times\Gh\} $
is the identity operator $\Op(\id)=\id$.

Note that (formally or whenever it makes sense),
if we denote the (right convolution) kernel of $\Op(\sigma)$ by $\kappa_x$,
that is, 
$$
\Op(\sigma)\phi(x)=\phi*\kappa_x,
\quad x\in G, \ \phi\in \cS(G),
$$
then it is given by
$$
\pi(\kappa_x)=\sigma(x,\pi).
$$
Moreover  the integral kernel of $\Op(\sigma)$ is 
$$
K(x,y)=\kappa_x(y^{-1}x),\quad\mbox{where}\quad
\Op(\sigma)\phi(x)=\int_G K(x,y) \phi(y)dy.
$$
We shall abuse the vocabulary and call $\kappa_x$ 
the kernel of $\sigma$, and $K$ its integral kernel.

\subsection{Pseudo-differential calculi on graded nilpotent Lie groups}

\subsubsection{Preliminaries on graded groups}
Graded groups are connected and simply connected 
Lie group 
whose Lie algebra $\fg$ 
admits an $\bN$-gradation
$\fg= \oplus_{\ell=1}^\infty \fg_{\ell}$
where the $\fg_{\ell}$, $\ell=1,2,\ldots$, 
are vector subspaces of $\fg$,
almost all equal to $\{0\}$,
and satisfying 
$[\fg_{\ell},\fg_{\ell'}]\subset\fg_{\ell+\ell'}$
for any $\ell,\ell'\in \bN$.
These groups are nilpotent. Examples of such groups are the Heisenberg group
 and, more generally,
all stratified groups (which by definition correspond to the case $\fg_1$ generating the full Lie algebra $\fg$); with a choice of basis or of scalar product on $\fg_1$, the latter are called Carnot groups. 

Graded groups are homogeneous in the sense of Folland-Stein \cite{folland+stein_82}
 when equipped with the  dilations
 given by the  linear mappings $D_r:\fg\to \fg$, 
$D_r X=r^\ell X$ for every $X\in \fg_\ell$, $\ell\in \bN$.
We may re-write the set of integers $\ell\in \bN$ such that $\fg_\ell\not=\{0\}$
into the increasing sequence of positive integers
 $\upsilon_1,\ldots,\upsilon_n$ counted with multiplicity,
 the multiplicity of $\fg_\ell$ being its dimension.
 In this way, the integers $\upsilon_1,\ldots, \upsilon_n$ become 
 the weights of the dilations and we have $D_r X_j =r^{\upsilon_j} X_j$, $j=1,\ldots, n$,
 on a basis $X_1,\ldots, X_n$  of $\fg$ adapted to the gradation.
 
 We denote the corresponding dilations on the group via
 $$
 rx = \exp (D_r X), \quad \mbox{for} \ x= \exp (X)\in G.
 $$
This leads to homogeneous notions for functions, distributions and operators. For instance, 
the homogeneous dimension of $G$ is the homogeneity of the Haar measure, that is, 
$Q:=\sum_{\ell\in \bN}\ell \dim \fg_\ell $;
and the differential operator $X^\alpha$ is homogeneous of degree 
$[\alpha]:=\sum_j \upsilon_j\alpha_{j}$.

\subsubsection{The symbolic pseudo-differential calculus on $G$}
\label{subsubsec_symbPDC}

In the monograph \cite{R+F_monograph},
the (Fr\'echet) space $S^m(G)$ of symbols of degree $m\in \bR$ on $G$ is defined
and the properties of  the corresponding space of operators $\Psi^m(G) = \Op(S^m(G))$ are studied.
Naturally, when $G$ is the abelian group $\bR$, 
the classes of symbols and of operators  are the ones due to H\"ormander. 

In the monograph, 
it is proved that $\Psi^*(G):= \cup_{m\in \bR} \Psi^m(G)$ is a symbolic pseudo-differential calculus  in the following sense: 
\begin{itemize}
\item $\Psi^*(G)$ is an algebra of operators, 
with an asymptotic formula for 
$\Op(\sigma_1)\Op(\sigma_2)=\Op(\sigma)$.
\item $\Psi^m(G)$ is adjoint-stable, i.e. 
$\Op(\sigma)^* =\Op(\tau) \in \Psi^m(G)$ when $\sigma\in S^m(G)$, 
with an asymptotic formula for $\tau$. 
\item
 $\Psi^*(G)$
contains the left-invariant  differential calculus as
$X^\alpha \in \Psi^{[\alpha]}(G)$. 
\item
 $\Psi^*(G)$
 contains the spectral calculus of the positive Rockland operators.
Note that in the context of graded groups,  the positive Rockland operators are the analogues of the elliptic operators 
\item
$\Psi^*(G)$ acts continuously on the  Sobolev spaces adapted to the graded groups with
$\Psi^m(G)\ni T : L^p_s(G)\hookrightarrow L^p_{s-m}(G)$.
\end{itemize}

\subsubsection{The classical pseudo-differential calculus on $G$}
\label{subsubsec_PsiclG}

Part of the paper \cite{FFPisa} is devoted to defining the notions of homogeneous symbols
and of classes $\dot S^m(G)$ of homogeneous symbols of degree $m$. 
Indeed, the dilations on the group $G$ induce an action of $\bR^+$ on the dual $\Gh$ via
\begin{equation}
\label{eq_rpi}	
r \cdot \pi (x) = \pi(r x), \qquad \pi\in \Gh, \ r>0, \ x\in G. 
\end{equation}

The homogeneous symbols are then measurable fields of operators  on $G\times \Sigma_1$ where
$$
\Sigma_1:=(\Gh / \bR^+) \setminus \{1_{\Gh}\}.
$$
is the analogue of the sphere on the Fourier side in the Euclidean case. 

This then allows us to consider symbols admitting a homogeneous expansion.
The space of operators in $\Psi^m(G)$ which admits a homogeneous expansion and whose integral kernel is compactly supported is denoted by $\Psi^m_{cl}(G)$. 
It is proved that $\Psi_{cl}^*(G):= \cup_{m\in \bR} \Psi^m_{cl}(G)$ is also a symbolic pseudo-differential calculus in the same sense as in Section \ref{subsubsec_symbPDC}. 
Furthermore, there is a natural notion of principal symbol associated to a symbol; the principal symbol is homogeneous by construction. 

Again, when $G$ is the abelian group $\bR^n$, 
this calculus is the well-known  classical pseudo-differential calculus, 
and the notion of principal symbol is the usual one. 

We set  $\Psi_{cl}^{\leq 0}(G):= \cup_{m\leq 0} \Psi^m_{cl}(G)$.
Depending on the context,  
the classical pseudo-differential calculus on $G$ may refer to 
the space of operators of any order in $\Psi_{cl}^*(G)$
or to the space of operators of non-positive orders $\Psi_{cl}^{\leq 0}(G)$. 

\subsubsection{The semi-classical pseudo-differential calculus on $G$}

The semi-classical pseudodifferential calculus 
 was presented in the context of groups of Heisenberg type in \cite{FFJST}, but in fact extends readily to any graded group $G$. 

We consider the class of symbols ${\mathcal A}_0$ of fields of operators defined on $G\times \Gh$ 
$$
\sigma(x,\pi)\in{\mathcal   L}(\cH_\pi),\;\;(x,\pi)\in G\times\Gh,
$$
that are of the form 
$$\sigma(x,\pi) = \cF_G \kappa_{x} (\pi),$$
where $\kappa_{x}(y)$ is smooth and compactly supported in $x$ while being Schwartz in $y$; more technically,   the map  
  $x\mapsto \kappa_{x}$ is in $C_c^\infty(G:\cS(G))$.
The group Fourier transform yields a bijection  from $C_c^\infty(G:\cS(G))$ onto $\cA_0$, and we equip $\cA_0$ with the Fr\'echet topology so that this mapping is an isomorphism of topological vector spaces. 

Let $\eps\in (0,1]$ be a small parameter.
For every symbol $\sigma\in \cA_0$, we consider the dilated symbol
obtained using the action of $\bR^+$ on $\Gh$, see \eqref{eq_rpi}, 
$$
\sigma^{(\eps)}:=
\{\sigma(x,\eps \cdot\pi) : (x, \pi)\in G\times \Gh\},
$$
and then the associated operator
$$
\Op^\eps (\sigma) :=  \Op (\sigma^{(\eps)}).
$$

As in the case of $\bR^n$ (see Section \ref{subsec_scmql}), 
this yields a (basic)
semi-classical calculus in the following sense:
\begin{itemize}
\item Each operator $ \Op^\eps (\sigma)$, $\sigma\in \cA_0$, is bounded on $L^2(G)$ with
$$
\| \Op^\eps (\sigma)\|_{\sL(L^2(G))} \leq  \|\sigma\|_{\cA_0}:= \int_{G} \sup_{x\in G}  |\kappa_{x}(y)|dy,
$$
where  $\kappa_x$ is the kernel of $\sigma$; $\|\cdot\|_{\cA_0}$ defines a continuous semi-norm on $\cA_0$.
\item The singularities  of the operators concentrate around the diagonal of the integral kernels as $\eps \to 0$:
$$
\forall N\in \bN \quad \exists C_N>0\quad \forall \eps\in (0,1] ,\ \sigma\in \cA_0\quad 
\|\sigma -  \cF_G \left( \kappa_{x} \chi (\eps \, \cdot)\right)\|_{\cA_0} \leq C {\eps}^{N}
$$
where $\chi \in C_c^\infty(G)$ is a fixed function identically equal to 1 on a neighbourhood of 0.
\item There is a calculus in the sense of expansions in powers of $\eps$ in $\sL(L^2(G))$
for products 
$\Op^{(\eps)}(\sigma_1)\Op^{(\eps)}(\sigma_2)$ and for  adjoints 
$\Op^{(\eps)}(\sigma)^*$; here $\sigma_1,\sigma_2,\sigma\in \cA_0$. 
\end{itemize}

\subsection{Operator-valued measures}
In Section \ref{sec_QL}, we explained why quantum limits in Euclidean or elliptic settings are often described with positive Radon measures on the spaces  of  symbols as these spaces are then commutative $C^*$-algebras.  
In the context of nilpotent Lie groups, the symbols are operator-valued, and we will see below that our examples of quantum limits  will then be described in terms of operator-valued measures  as introduced in \cite{FFchina,FFPisa}. Let us recall the precise definition of this notion:

\begin{definition}
\label{def_gammaGamma}
	Let $Z$ be a complete separable metric space, 
	and let $\xi\mapsto \cH_\xi$ be a measurable field of complex Hilbert spaces of $Z$.
\begin{itemize}
\item 
	The set 
	$ \widetilde{\mathcal M}_{ov}(Z,(\cH_\xi)_{\xi\in Z})$
	is the set of pairs $(\gamma,\Gamma)$ where $\gamma$ is a positive Radon measure on~$Z$ 
	and $\Gamma=\{\Gamma(\xi)\in {\mathcal L}(\cH_\xi):\xi \in Z\}$ is a measurable field of trace-class operators
such that
$$\|\Gamma d \gamma\|_{\mathcal M}:=\int_Z{\rm Tr}_{\cH_\xi} |\Gamma(\xi)|d\gamma(\xi)
<\infty.
$$
Here ${\rm Tr}_{\cH_\xi} |\Gamma(\xi)|$ denotes the standard trace of the trace-class operator $ |\Gamma(\xi)|$ on the separable Hilbert space $\cH_\xi$. 	
	
 \item 
	Two pairs $(\gamma,\Gamma)$ and $(\gamma',\Gamma')$ 
in $\widetilde {\mathcal M}_{ov}(Z,(\cH_\xi)_{\xi\in Z})$
are {equivalent} when there exists a measurable function $f:Z\to \mathbb C\setminus\{0\}$ such that 
$$d\gamma'(\xi) =f(\xi)  d\gamma(\xi)\;\;{\rm  and} \;\;\Gamma'(\xi)=\frac 1 {f(\xi)} \Gamma(\xi)$$ for $\gamma$-almost every $\xi\in Z$.
The equivalence class of $(\gamma,\Gamma)$ is denoted by $\Gamma d \gamma$,
and the resulting quotient set is 
 denoted by ${\mathcal M}_{ov}(Z,(\cH_\xi)_{\xi\in Z})$.

\item 
A pair $(\gamma,\Gamma)$ 
in $ \widetilde {\mathcal M}_{ov}(Z,(\cH_\xi)_{\xi\in Z})$
 is {positive} when 
$\Gamma(\xi)\geq 0$ for $\gamma$-almost all $\xi\in Z$.
In  this case, we may write  $(\gamma,\Gamma)\in  \widetilde {\mathcal M}_{ov}^+(Z,(\cH_\xi)_{\xi\in Z})$, 
and $\Gamma d\gamma \geq 0$ for $\Gamma d\gamma \in {\mathcal M}_{ov}^+(Z,(\cH_\xi)_{\xi\in Z})$.
\end{itemize}
\end{definition}

By convention and if not otherwise specified,  a representative of the class $\Gamma d\gamma$ is chosen such that ${\rm Tr}_{\cH_\xi} \Gamma=1$. 
In particular, if $\cH_\xi$ is $1$-dimensional, $\Gamma=1$ and  $\Gamma d\gamma$ reduces to the measure $d\gamma$.
One checks readily that $\mathcal M_{ov} (Z,(\cH_\xi)_{\xi\in Z})$ equipped with the norm $\| \cdot\|_{{\mathcal M}}$ is a Banach space.

\medskip 

When the field of Hilbert spaces is clear from the setting, 
we may write 
$$
\mathcal M_{ov} (Z) = \mathcal M (Z,(\cH_\xi)_{\xi\in Z}),
\quad
\mbox{and}\quad
\mathcal M_{ov}^+ (Z) = \mathcal M^+ (Z,(\cH_\xi)_{\xi\in Z}),
$$
for short.
For instance, if $\xi\mapsto \cH_\xi$ is given by $\mathcal H_\xi=\mathbb C$ for all $\xi$, 
then $\mathcal M (Z)$ coincides with the space of finite Radon measures on $Z$.
Another example is when $Z$ is of the form $Z=Z_1 \times \widehat G$ where $Z_1$ is a complete separable metric space, and $\mathcal H_{(z_1,\pi)}= \cH_\pi$, where
 the Hilbert space $\cH_\pi$ is associated with the representation $\pi \in \widehat G$. 

\subsection{Micro-local defect measures on graded Lie groups}

In \cite{FFPisa}, the following analogue to Theorem \ref{thm_MDM} is proved 
in the setting of graded groups. It uses the classical pseudo-differential calculus and the sphere $\Sigma_1$ of the dual as mentioned in Section \ref{subsubsec_PsiclG} and the notion of operator-valued measure (see 
Definition \ref{def_gammaGamma}). 

\begin{theorem}
\label{thm_MDMG}
Let $\Omega$ be an open subset of $G$. 
Let $(f_j)_{j\in \bN}$ be a bounded sequence in $L^2(\Omega,loc)$ converging weakly to 0.  
Then there exists a subsequence $(j_k)_{k\in \bN}$ and an operator-valued measure 
$\Gamma d\gamma \in \cM_{ov}^+(G\times \Sigma_1 )$
 such that 
$$
(Af_j,f_j)_{L^2} \longrightarrow_{j=j_k, k\to \infty} 
\int_{\Omega \times \Sigma_1}
\tr \left(\sigma_0 (x,\dot \pi) \ \Gamma(x,\dot \pi) \right)
d  \gamma(x,\dot\pi) \, ,
$$
holds for all classical pseudo-differential operator $A\in \Psi^{\leq 0}_{cl}(G)$, $\sigma_0$ denoting its principal symbol. 
\end{theorem}

The proof of Theorem \ref{thm_MDMG} given in \cite{FFPisa} follows the same ideas as the ones presented in Section \ref{subsec_viewptQL} with the adaptations that come from dealing with a more non-commutative $C^*$-algebra of symbols.

Examples of micro-local defect measures developed in \cite{FFPisa} include 
\begin{itemize}
\item an  $L^2$-concentration in space,
\item an  $L^2$-concentration in oscillations using matrix coefficients of representations.  
\end{itemize} 
An application to compensated compactness is also deduced. 
It would be interesting to relate this to the works by B. Franchi and his collaborators \cite{Baldi+Franchi+Tesi08, Baldi+Franchi+Tesi08b,Franchi,Franchi+Tchou+Tesi} 
 on compensated compactness on the Heisenberg group. 

\subsection{Semi-classical measures on graded Lie groups}

In \cite{FFchina}, 
the semi-classical analysis developed on $G$ yields the same property of existence of (group) semi-classical measures:

 \begin{theorem}
 \label{thm_scmG}
 Let $(f_\eps)_{\eps>0}$ be a bounded family in $L^2(G)$. 
 Then there exists a sequence $\eps_k$, $k\in \bN$ with $\eps_k\to 0$ as $k\to \infty$, 
 and an operator-valued measure 
$\Gamma d\gamma \in \cM_{ov}^+(G\times \Gh )$
 satisfying
 $$
 \forall \sigma\in \cA_0
 \qquad
(\Op_\eps (\sigma)f_\eps,f_\eps)_{L^2} \longrightarrow_{\eps=\eps_k, k\to \infty} 
\int_{\Omega \times \Gh}
\tr \left(\sigma (x, \pi) \ \Gamma(x, \pi) \right)
d  \gamma(x,\pi) .
$$
 \end{theorem}
 
The (group) semi-classical analogues of the (group) micro-local defect measures for  an $L^{2}$-concentration in space and an $L^{2}$-concentration in oscillations is also given in \cite{FFchina} in the context of the groups of Heisenberg type; naturally, the former holds on any graded group. 

In \cite{FFJST}, we prove an analogue of  the application given in Section \ref{subsec_app} but for the sub-Laplacian on any group of Heisenberg type.
We obtain a description of the $t$-dependent group semi-classical measures corresponding to the solutions to the Schr\"odinger equations, 
and therefore of their weak limits after taking the $x$-marginals. 
However, there is not one threshold $\tau=1$ as in the Euclidean case, but two, namely $\tau=1$ and $\tau=2$. 
More precisely, the semi-classical measures and the weak limits can be written  into two parts:
\begin{itemize}
\item	
 one with a Euclidean behaviour and threshold $\tau=1$, 
and
\item  one with threshold $\tau=2$. 
\end{itemize}
With our methods, this comes from the splitting of the unitary dual $\Gh$ into the following two subsets:
\begin{itemize}
\item the subsets of
 infinite dimensional representations (for instance realised as the Scr\"odinger representations), and
 \item the subset of finite dimensional representations, in fact of dimension one and given by the (abelian or Euclidean) characters of the first stratum. 
 \end{itemize}
This splitting is also present in other works that do not involve representation theory; see for instance \cite{BS} about the Grushin-Schr\"odinger equation and~\cite{Zeld97,CdVHT} about sublaplacians on contact manifolds.
In fact, this phenomenon of slower dispersion  than in Euclidean settings has already been observed for other sub-Riemannian PDEs, see e.g. \cite{BGX,hiero,BFG2}.

\subsection{Future works}

The tools developed so far in \cite{FFchina,FFPisa,FFJST} 
can be adapted to (graded) nilmanifolds along the lines of \cite{Fermanian+Letrouit}. 
Nilmanifolds are quotients of nilpotent Lie groups by a discrete subgroup. 
When the subgroup is also co-compact, this results in a compact manifold which is locally given by the group. This provides an excellent setting for the applications  to PDEs of the theory developed  in \cite{FFchina,FFPisa,FFJST} .  

The extension to sub-Riemannian manifolds will certainly be more difficult. However, given the recent progress in groupoids 
on filtered manifolds \cite{vanErp,choi+ponge,vanErp+Y},
the author feels confident that the semi-classical and micro-local analysis already developed on graded groups will be transferable to the setting of equiregular sub-Riemannian manifolds in the near future.

\bibliographystyle{alpha}

\end{document}